\documentclass{lms-mod}
\usepackage{amsmath,amsfonts,amscd,amssymb,amsxtra,color}
\usepackage{mathrsfs}
\newtheorem{theorem}{Theorem}[section] 
\newtheorem{corollary}[theorem]{Corollary}
\newtheorem{proposition}[theorem]{Proposition}

\newnumbered{assertion}{Assertion}    
\newnumbered{conjecture}{Conjecture}  
\newnumbered{definition}{Definition}
\newnumbered{hypothesis}{Hypothesis}
\newnumbered{remark}{Remark}
\newnumbered{note}{Note}
\newnumbered{observation}{Observation}
\newnumbered{problem}{Problem}
\newnumbered{question}{Question}
\newnumbered{algorithm}{Algorithm}
\newnumbered{example}{Example}
\newunnumbered{notation}{Notation}

\title[Spectral Properties of the Ruelle Operator]
{Spectral Properties of the Ruelle Operator for Product Type Potentials on Shift Spaces}

\author{L. Cioletti, M. Denker, A.~O. Lopes and M. Stadlbauer}

\classno{28Dxx (primary), 37D35, 37B10, 82B05 (secondary)}

\extraline{
L. Cioletti, O.A. Lopes and M. Stadlbauer are partially supported by CNPq. 
M. Denker was partially funded by CAPES  
Grant 158/2012--Pesquisador Visitante Especial.
}

\begin{document}
\maketitle

\begin{abstract}
	We study a  class of potentials $f$ on  one sided full shift spaces over
	finite or countable alphabets, called potentials of product type.
	We obtain explicit formulae for the leading eigenvalue, the eigenfunction
	(which may be discontinuous) and the eigenmeasure of the Ruelle operator.
	The uniqueness property of these quantities is also discussed and it is
	shown that there always exists a Bernoulli equilibrium state even if $f$
	does not satisfy Bowen's condition.
	
	We apply these results to potentials
	$f:\{-1,1\}^\mathbb{N} \to \mathbb{R}$ of the form
	$$
	f(x_1,x_2,\ldots)
	=
	x_1 + 2^{-\gamma} \, x_2 + 3^{-\gamma} \, x_3 + ...+n^{-\gamma} \, x_n + \ldots
	$$
	with $\gamma >1$. For $3/2 < \gamma \leq 2$, we obtain the existence of two different 
	eigenfunctions. Both functions are (locally) unbounded and exist a.s. 
	(but not everywhere) with respect to the eigenmeasure and the measure 
	of maximal entropy, respectively.
\end{abstract}

\section{Introduction}

The theory of Gibbs states in physics and mathematics led to the notion 
of  the pressure function and its variational formula for 
dynamical systems (Ruelle 1967, \cite{Ruelle1967} and Walters 1975, \cite{Wal1975}). 
Since then a variety of results has been published to clarify 
existence and uniqueness of equilibrium states maximizing the pressure, 
and this note is in the same spirit.

The classical condition for uniqueness of the equilibrium state 
requires summable variations  and was relaxed by  Bowen (\cite{Bowen1974}) 
using a condition which is named after him. 
This  has been further investigated by Walters 1978 (\cite{Wal78}) 
who introduced a slightly stronger condition, which is referred to 
as Walters' condition, see also \cite{Bousch2001}. 
Yuri in 1998 (\cite{Yur98}) coined the term weak bounded variation 
and also showed uniqueness.  
For  many classes of maps on compact spaces uniqueness has been 
proved as well, as a recent examples for this,  
Climenhaga and Thompson in 2013 (\cite{Climenhaga2013}) used a 
restricted  Bowen condition, and Iommi and Todd (\cite{Iommi2013}) 
studied the existence of phase transitions for 
grid potentials (see \cite{Markley1982}) on full shift spaces. 
We finally mention Sarig's work in  2001 (\cite{Sarig2001}) 
which opened a new field of studying this question 
on countable subshifts (the non-compact case) 
using Gurevic' pressure, or, for a more general approach to 
pressure, the notion introduced by Stratmann and 
Urba\'nski in 2007 (\cite{SU07}).

In expansive dynamical systems an equilibrium state always exists, 
leading to the problem of uniqueness and
continuity properties of the density of the equilibrium state 
with respect to {\it canonical} measures.  
These canonical measures may be defined as conformal measures 
(in many cases the eigenmeasure of the Ruelle operator associated to
the normalized potential or simply Gibbs measures on shift spaces) 
or  - as we show below - 
product measures (for example a Bernoulli measure on shift spaces).

In this note we deal with a dynamical system $T:X\to X$ where $T$ denotes the shift transformation on $X=\mathcal A^{\mathbb N}$, where $\mathcal A$ is a finite or countable set, called the alphabet of the dynamics, and where $X$ is equipped with the product topology of pointwise convergence and the associated Borel $\sigma$-field.
We consider potential functions
$$ f:X\to \mathbb R$$
which can be written in the form
$$ f(x)= \sum_{n=1}^\infty f_n(x_n),\qquad x=(x_n)_{n\in \mathbb N}\in X$$
and call these functions of product type 
(see Section \ref{sec:producttype}), where $f_n:\mathcal A\to \mathbb R$ 
are fixed functions so that the sum converges.
Although $f$ is given by a sum (possibly a series) the terminology product type 
is convenient because the function $g=\exp(f)$ appearing in the Ruelle operator   
can be naturally represented by a product (possibly an infinity product).

Given a function $f:X\to \mathbb R$ the Ruelle operator
\begin{equation}\label{eq:ruelle}
\mathcal L_f\phi(x)= \sum_{T(y)=x} \phi(y)\exp{f(y)}
\end{equation}
acts on bounded measurable functions if $\mathcal L_f(1)(x)<\infty$ for all $x\in X$.

The initial motivation for the present note was to show the existence of  positive measurable eigenfunctions of  $\mathcal{L}_f$ and obtain criteria for its continuity (see \cite{Wal01} for details). In Section \ref{notB}, we show that, for a continuous potential $f$ less regular than a Bowen potential, the eigenfunction might oscillate between $0$ and $\infty$ on any open set  (see  Theorem \ref{theorem:oscilation}).

If $f$ is of product type, the function $g=e^f$ appearing in the Ruelle operator  (\ref{eq:ruelle})  has indeed a product structure.  It is not hard to see that $\mathcal{L}_f$ and its dual act on functions with a product structure and product measures, respectively.
These basic observations permit  explicit representations of eigenfunctions, conformal measures and equilibrium states (which are of possible  interest in connection with computer experiments or applications in  mathematical physics). There are examples of potentials of product type which belong to Bowen's and Walters' class (see  \cite{Wal01,Wal05}), but also examples having less regularity properties than potentials in these two classes.

We consider the following classes of potentials of product type.
We say that $g=e^f$ is \emph{$\ell_1$-bounded} if $(\|f_k\|_\infty)_{k\ge 2}\in \ell_1$, i.e. $\sum_{k=2}^\infty \|f_k\|_\infty < \infty $ and is
\emph{summable} if $\sum_{a \in \mathcal{A}} \exp({f_1(a)})< \infty$. Moreover, $g$ is a \emph{balanced potential}, if $\sum_{a \in \mathcal{A}} f_k(a)=0$ for all $k\ge 1$. Note that the first condition is equivalent to the condition that  $g(x_1, x_2 \ldots)/\exp({f_1(x_1)})$ is uniformly bounded. Combined with summability, this implies that $\|\mathcal{L}_{\log g}(1)\|_\infty < \infty$, independently of $\mathcal A$ being finite or not. A balanced potential may be considered as a kind of
normal form for potentials of product type.

These conditions on potentials of product type  can be used to describe the properties of the corresponding Ruelle operator. We obtain the following results for the existence of conformal and equilibrium measures under rather weak assumptions. If $\|g\|_\infty < \infty$, then there is an explicitly given product measure which is
$1/g$-conformal (Theorem \ref{theo:existence}). Furthermore, if $g$ is summable and $\ell_1$-bounded, then there exists an explicitly given Bernoulli measure which is an equilibrium state.

In order to obtain uniqueness of these measures, we have to impose Bowen's condition.  
We say that $g=e^f$ is in \emph{Bowen's class}  
if $\log g$ is of locally bounded distortion
(see Proposition \ref{prop_Walters_Bowen_Yuri}). 
That is, there exists $k \in\mathbb N$, referred to as index, such that
\[
\sum_{m=k}^\infty  
\sum^\infty_{n=m}  
\sup \{|f_n (a) - f_n(b)| : a,b \in \mathcal{A}\} 
< \infty.
\]
If Bowen's condition holds for $k=2$, observe that a summable and balanced potential automatically is locally bounded.
Under these assumptions, we show that there exists an explicitly given continuous eigenfunction of $\mathcal{L}_{\log g}$ (Theorem \ref{theo:eigenfunction-abstract}) and, if $\mathcal A$ is finite, that the conformal measure and the equilibrium state are unique (Theorems \ref{theo:uniqueness}, \ref{theo:equilibrium}).

Beyond Bowen's condition, the situation is very different.   If $\mathcal{A}$ is finite and for some $k$,
\begin{equation}\label{eq:intermediate}
\sum_{i=k}^\infty \max_{a \in \mathcal{A}} \bigg({ \sum_{j=i}^\infty }\log g_j(a)\bigg)^2< \infty,
\end{equation}
there are three canonical measures, first the conformal measure 
$\mu$ for $1/g$, secondly the equilibrium measure $\tilde \mu$ and 
last the measure of maximal entropy $\rho$. 
All three measures are Bernoulli (i.e. the coordinate process is independent) 
and $\mu$ and $\tilde{\mu}$ are absolutely continuous with respect to each other.
Moreover, there exist functions $h_\mu \in L^1(X,\mu)$ and $h_\rho \in L^1(X,\rho)$ 
which may exist only almost surely, but these functions are 
eigenfunctions for the action of the operator on $L^1(X,\mu)$ and $L^p(X,\rho)$ (for $1\leq p < \infty$), respectively. 
The relationship between $h$ and the equilibrium measure is
explained by ergodicity of a natural operator on $L^1(X,\rho)$ defined by $\tilde \mu$.

In order to illustrate the results we will study an explicit example in 
Section \ref{exa}. In there, we consider the potential $f:\{-1,1\}^\mathbb{N} \to \mathbb{R}$ of the form
$$ f(x_1,x_2,...) = x_1 + 2^{-\gamma} \, x_2 + 3^{-\gamma} \, x_3 + ...+n^{-\gamma} \, x_n + ...$$
which is a summable, locally bounded and balanced potential for $\gamma >1$.
If $\gamma>2$, then $e^f$ is in Bowen's class, and for $3/2 < \gamma  \leq 2$, condition \eqref{eq:intermediate} is satisfied.
For the latter case, we obtain that $h_\mu$ and $h_\rho$ are locally unbounded and therefore discontinuous. Furthermore, for  $1<\gamma\leq 3/2$, 
these eigenfunctions do not exist and the measures  $\mu$, $\tilde \mu$ and $\rho$ 
are pairwise singular.

The paper is structured as follows.
In Section \ref{sec:topology}, we recall the regularity classes of  Bowen,
Walters and Yuri adapted to the setting of potentials of product type.
In our setting, the classes of Bowen and Walters coincide, and in particular,
the existence of conformal measures and continuous eigenfunctions
for finite $\mathcal A$ could also be obtained by
results in \cite{Wal01}.
For completeness, we give conditions for a
potential of product type to be in Yuri's class, although we do not prove results under this
regularity hypothesis in this paper.
This is due to the fact that the results by Yuri rely on a tower construction whose associated potential is of bounded variation - or, from a more abstract viewpoint,
on the existence of an isolated critical set or isolated indifferent fixed points.

In Section \ref{sec:producttype}, we then provide a very general condition for the existence of a conformal measure and a condition for uniqueness. These results essentially rely on the observation that the Radon-Nikodym derivative of a measure of product type is a function of product type, and an ergodicity argument, respectively.
In Section \ref{sec:Eigenfunctions of product type}, we explicitly construct eigenfunctions and equilibrium states, including the existence only $\rho$-almost everywhere when Bowen's condition is not satisfied and where $\rho$ denotes the measure of maximal entropy. This is extended in the following section to the action of the Ruelle operator on $L^1(X,\rho')$ for certain product measures and a condition for the uniqueness of $h$ is given. Section \ref{exa} is then dedicated to the analysis of the above mentioned example.

\section{Regularity classes of potentials} \label{sec:topology}
In order to adapt the conditions by Bowen, Walters and Yuri to functions of product type we begin specifying a metric on $X=\mathcal A^{\mathbb N}$.
For $(x_n), (y_n) \in X$, let
$$ d(x,y)= 2^{-\max \{n: x_k =y_k \forall k \leq n\}}. $$
As it is well known, $d$ generates the product topology of pointwise 
convergence and $(X,d)$ is a complete metric space which is compact  if and only if $\mathcal{A}$ is finite. 
The cylinder sets form a basis of this topology, where,
for a $k$-word $(x_1, \ldots , x_k) \in \mathcal{A}^k$, the associated  cylinder set is
defined by $[x_1, \ldots , x_k] := \{(y_n)_{n\ge 1} \in X: y_i =x_i \forall i=1,\ldots, k\}$.

The shift on $X$ is defined by $T: X \to X, (x_1,x_2,\ldots)\mapsto (x_2,\ldots)$ and, as it is well known, is a continuous transformation which expands distances by 2. In order to put emphasis on the underlying topology and Borel $\sigma$-algebra, we will refer to $(X,T)$ as a \emph{topological Bernoulli} shift over the alphabet $\mathcal{A}$.

Using a slightly different notation as in \cite{Wal01}, for a function $\phi:X\to \mathbb R$  we let
$$\mathrm{var}_{n}(\phi) :=  \sup \{ |\phi(x)-\phi(y)|: d(x,y)\leq 2^{-n}\} $$
denote the variation of $\phi$ over cylinders of length $n$.
Then $\phi$ has summable variations  (\cite{Wal75a}) if
$$ \sum_{n=1}^\infty \mathrm{var}_n(\phi)<\infty.$$
To simplify the notation we write $S_n(\phi)=\phi+...+\phi\circ T^{n-1}$. 
In the sequel we define some regularity classes in terms of the decay 
of $\mathrm{var}_{n}(\cdot)$. 
We say that a function $\phi: X \to \mathbb R$ belongs to
\begin{enumerate}
	\item  {Walters' class} (\cite{Wal01}) if $ \lim_{k \to \infty} \sup_{n\in\mathbb{N}}\ \mathrm{var}_{n+k}(S_n(\phi)) =0$,
	\item  {Bowen's class} (\cite{Bowen1974}) if $\exists \;k\in\mathbb{N} $ such that
	$ \sup_{n\in\mathbb{N}}\ \mathrm{var}_{n + k}(S_n(\phi)) < \infty $,
	\item {Yuri's class} (\cite{Yur98}) if
	$ \lim_{n \to \infty} \frac{1}{n}\mathrm{var}_{n}(S_n(\phi)) =0$.\end{enumerate}
It has been remarked in \cite{Wal01} that the definition of Bowen's class given 
here is equivalent to Bowen's original definition.
Observe that for shift spaces, Walters' condition is equivalent to equicontinuity of the family $\{S_n(\phi)): n\ge 1\}$, whereas Bowen's condition provides a uniform local bound on the local distortion of $(S_n(\phi))_{n\ge 1}$. Yuri's condition is also known as weak bounded variation (\cite{Yur98}). We now deduce necessary conditions  for functions of product type to belong to these classes.  Assume that $f:X \to \mathbb R$ is of the form
\[ f((x_n)_{n\ge 1}) = \sum_{n=1}^\infty f_n(x_n), \]
where $(f_n: \mathcal A \to \mathbb R)_{n\ge 1}$ is a sequence such that $\sum_n  f_n(x_n)$ converges for all $x =(x_n)_{n\ge 1}\in X$ and set
\[ v_{n}(f) := \sup \{|f_n (a) - f_n(b)| : a,b \in \mathcal A\}, \quad s_n(f) := \sum_{k>n} v_k(f).\]

\begin{proposition}\label{prop_Walters_Bowen_Yuri} For $f$  of product type as above, the following holds.
	\begin{enumerate}
		\item If \, $\sum_{n=1}^\infty s_n(f)<\infty$ then $f$ has summable variation.
		\item  If \, $\sum_{n=k}^\infty s_n(f) < \infty$ for some $k \in \mathbb N$, then $f$ belongs to Bowen's and Walters' class.
		\item  If \, $\lim_{m \to \infty}  \frac{1}{m} \sum_{n=1}^m s_n(f)=0$, then $f$ belongs to Yuri's class.
	\end{enumerate}
\end{proposition}

\begin{proof} For $x = (x_n)_{n\ge 1}$, $y = (y_n)_{n\ge 1}$ with $x_j=y_j$ for all $j \leq m+k$, it follows that
	\begin{align*}
	\left|S_m(f)(x) - S_m(f)(y) \right| & = \left|\sum_{j=0}^{m-1}  \sum_{n=1}^\infty  f_n(x_{n+j}) - f_n(y_{n+j})\right| \\
	& =   \left| \sum_{j=0}^{m-1}  \sum_{n=m-j+k+1}^\infty   f_n(x_{n+j}) - f_n(y_{n+j})\right|\\
	& \leq  \sum_{l=1}^{m}  \sum_{n=l+k+1}^\infty   v_{n}(f) = \sum_{l=1}^{m} s_{l+k}(f).
	\end{align*}
	Hence, $\mathrm{var}_{m+k}(S_m(f)) \leq \sum_{l=1}^{m} s_{l+k}(f) \leq \sum_{l> k } s_{l}(f)$. 
	Assertions 2 and 3 easily follow from this estimate. The assertion 1 is shown similarly.
\end{proof}

\begin{example} Assume that  $\|f_n\|_\infty \ll n^{-\gamma}$  for  some $\gamma >1$, where  $a_n \ll b_n$ stands for the existence of $C>0$ with $a_n \leq C b_n$ for all $n \in \mathbb N$. As $\gamma>1$, it follows that $\sum_n \|f_n\|_\infty < \infty$. Moreover, the estimate $v_{n}(f) \leq 2 \|f_n\|_\infty \leq 2 n^{-\gamma}$ implies that $s_n(f) \ll n^{1-\gamma}$. In particular, $\sum_{n=1}^m s_n(f)  \ll n^{2-\gamma}$.
	Hence, if $\gamma >2$, then $f$ has summable variation and is in Bowen's and Walters' class, and if $\gamma >1$, then $f$ is in Yuri's class.
\end{example}

\begin{example} \label{ex:simplified-dyson}
	In order to see that this classification through $\gamma$ is sharp, we consider the specific example $f:\{-1,1\}^\mathbb{N} \to \mathbb{R}$ of the form $ f(x) = \sum_n x_n n^{-\gamma}$.
	Then, for $x=(x_n)_{n\ge 1}$ and $y=(y_n)_{n\ge 1}$ with $x_j=y_j$ for all $j \leq m+k$ and $x_j =1$ and  $y_j = - 1$ for all $j > m+k$,
	one obtains as in the proof of Proposition \ref{prop_Walters_Bowen_Yuri} that, for $\gamma \neq 2$,
	\begin{align*}
	S_m(f)(x) - S_m(f)(y) & = \sum_{l=1}^{m}  \sum_{n>l+k}^\infty   f_n(1) - f_n(-1) = 2 \sum_{l=1}^{m}  \sum_{n>l+k}
	n^{-\gamma}\\
	& \gg   \sum_{l=1}^m (l+k+1)^{1-\gamma} \gg \left|  (k+2)^{2-\gamma} - (m+k+2)^{2-\gamma}\right|.
	\end{align*}
	By the same argument, it follows that $S_m(f)(x) - S_m(f)(y) \gg \log(m+k+2) -\log (k+2)$ for $\gamma = 2$.
	Hence, for this particular choice of $f$, it follows that $f$ is in Bowen's or Walters' class if and only if $\gamma>2$. Furthermore, $f$ is in Yuri's class if and only if $ \gamma>1$.
\end{example}

\section{Conformal measures of product type} \label{sec:producttype}

\subsection{Existence}
Conformal measures are used to denote the existence  of probability 
measures $\mu$ with a prescribed Jacobian $J = d\mu \circ T / d\mu$. 
In this section we study their existence and uniqueness for a 
given potential $f$ of product type, where the Jacobian 
is given by $J= e^{-f}$. Hence if $g:X\to \mathbb R_+$ is a 
given positive  function (also called a potential), 
$f =\log g$ is the potential for the associated Ruelle operator 
$\mathcal{L}_f$   (see below),
and $g$ is said to be of product type if the associated $f$ is of this type,
in particular, $g$ can be written in the form
$g(x)=\prod_{n=1}^\infty g_n(x_n)$ ($x=(x_n)_{n\ge 1}$),
where the $g_n$ are uniquely determined up to non-zero constants. 
In analogy to product type functions,
we also call a product measure $\mu=\otimes_{i=1}^\infty \mu_i$ on
$X=\mathcal A^{\mathbb N}$  a measure of product type,
where $\mu_i$ are probability measures on $\mathcal A$.
Such probability measures are sometimes called Bernoulli measure.
These product measures are uniquely defined by their  values on cylinders:
$$
\mu([a_1,\ldots, a_n]) =\prod_{i=1}^n \mu_i(a_i)\qquad a_1,...,a_n\in\mathcal A.
$$
Recall from \cite{DU91} that a Borel probability measure
$\mu$ on $(X,\mathcal{B})$ is $\phi$-conformal if there exists $\lambda >0$,
$$
\mu(T(A)) =  \lambda \int_A \phi\, d\mu
$$
for all measurable sets $A$ such that the shift map $T:X \to X$ restricted to $A$ is injective. If the Ruelle operator  $\mathcal{L}_{- \log \phi}$ acts on continuous functions its dual operator  also   acts on finite signed measures, and it is well known that  a measure $\mu$ is $\phi$-conformal if and only if
$\mathcal{L}^\ast_{- \log \phi}(\mu) = \lambda \mu$, for some $\lambda>0$. Also note that $\lambda$ usually is equal to the spectral radius of $\mathcal{L}_{- \log \phi}$.

\begin{theorem} \label{theo:existence}
	Let $(X,T)$ be a topological Bernoulli shift over a finite
	or countable alphabet $\mathcal A$ and let $g=\prod_{n=1}^\infty g_n$ be a potential
	of product type.
	\begin{enumerate}
		\item \label{cor:conf_measure_formula}
		There exists at most one conformal measure $\mu$
		of product type for $g$ which is positive on open sets. This measure $\mu$ is given by
		\begin{equation}\label{eq:conformal}
		\mu_n(a) = \left( \sum_{b \in \mathcal{A}} \prod_{i=1}^n \frac{g_i(a)}{g_i(b)} \right)^{-1} \hbox{ for all } n \in \mathbb N, \; a \in \mathcal{A}.
		\end{equation}
		\item If $\inf_{x\in X} g(x) >0$, then a conformal
		measure of product type for $g$
		exists and is positive on open sets.
	\end{enumerate}
\end{theorem}

\begin{proof} We begin with the proof of the first assertion.
	Let $\mu=\otimes_{i=1}^\infty \mu_i$ be a product measure
	which is positive on open sets,
	in particular on each cylinder set. In order that
	it is conformal for $g$, that is
	$$
	\mu(T[a_1,\ldots,a_n]) = \mu([a_2,\ldots,a_n]) = \lambda \int_{[a_1,\ldots,a_n]} g(x)\mu(dx)
	$$
	for every cylinder set $[a_1,\ldots,a_n]$,
	it is necessary and sufficient that
	\begin{equation}\label{eq:1}
	1 =\mu(T[a])= \lambda \int_{[a]} g_1(x_1) \mu_1(dx_1) \prod_{i=2}^\infty g_i(y)\mu_i(dy)
	\end{equation}
	for some $\lambda>0$ and
	
	\begin{eqnarray}\label{eq:2}
	\nonumber \mu_1(a_2)\ldots\mu_{n-1}(a_{n}) &=& \mu(T([a_1,\ldots,a_n]))\\
	&=& \lambda \prod_{i=1}^n g_i(a_i)\mu_i(a_i) \prod_{i=n+1}^\infty \int g_i(y)\mu_i(dy)
	\end{eqnarray}
	for any $a_1,\ldots,a_n\in \mathcal{A}$.
	Varying $a\in \mathcal{A}$
	in equation (\ref{eq:1}) yields
	$$
	g_1(a)\mu_1(a) = g_1(b)\mu_1(b)
	$$
	and hence
	\begin{equation}\label{eq:conformal2}
	\mu_1(a) =\left( g_1(a)\sum_{b\in \mathcal{A}} \frac 1{g_1(b)} \right)^{-1}.
	\end{equation}
	The similarly equations (\ref{eq:2}) yield
	\begin{equation}\label{eq:conformal3}
	\mu_n(a)= \frac{\mu_{n-1}(a)}{g_n(a)} \left(\sum_{b\in \mathcal{A}} \frac{\mu_{n-1}(b)}{g_n(b)}\right)^{-1}.
	\end{equation}
	It follows that the conformality equalities (\ref{eq:1}) amd (\ref{eq:2}) uniquely determine
	the conformal measure (which is positive on open sets and a product measure),
	hence the uniqueness of $\mu$. Moreover, by (\ref{eq:2}),
	\[ \frac{\mu_n(a)}{\mu_n(b)} = \frac{g_n(b)}{g_n(a)} \cdot \frac{\mu_{n-1}(a)}{\mu_{n-1}(b)}. \]
	Hence, the first part of the theorem follows  by induction.\\
	
	For the proof of the second part, note that the uniform lower bound on $g$ is equivalent to
	\begin{equation*}\label{eq:3}
	\sum_{i=1}^\infty\log  \|g_i^{-1}\|_\infty^{-1}>-\infty.
	\end{equation*}
	Hence, for any sequence of measures $\mu_i$ on $\mathcal  A$,
	$$
	\int_{X} g(x) \prod_{i=1}^\infty \mu_i(dx) = \prod_{i=1}^\infty \int_{\mathcal{A}} g_i(u) \mu_i(du)
	\ge \prod_{i=1}^\infty \|g^{-1}_i\|_\infty^{-1} >0.
	$$
	Hence the equations (\ref{eq:1}) and (\ref{eq:2})
	show that the conformal product measure is well defined
	and positive on open sets.
\end{proof}

Due to the constructive proof above,
it is possible to obtain explicit expressions for the measure and the
associated parameter $\lambda$.

\begin{corollary} \label{cor:lambda}
	If $\inf_{x} g(x) >0$, then for every
	$t\in \mathbb R$, the function $g(t)= g^t$ satisfies $\inf_{x} g^t(x) >0$
	as well and the conformality parameter $\lambda_t$ satisfies
	$$
	\lambda_t=  \sum_{c\in \mathcal A}
	\frac{1}{\prod_{i=1}^\infty g(t)_i(c)}
	$$
	for all $t$ where the denominator does not vanish.
\end{corollary}
\begin{proof} We may put $t=1$. Inserting (\ref{eq:conformal})  into equation (\ref{eq:1}) yields
	$$ 1= \lambda \left(\sum_{b\in \mathcal A} g_1(b)^{-1}\right)^{-1}\prod_{i=2}^\infty \int g_i(u) \mu_i(du).$$
	Now by equation (\ref{eq:conformal})
	$$
	\int g_n d\mu_n=\sum_{b\in \mathcal A} g_n(b) \mu_n(b)
	=\left(\sum_{b\in \mathcal A} \frac {\mu_{n-1}(b)}{g_n(b)}\right)^{-1}
	$$
	and by bachward induction over $m$
	$$
	\int g_nd\mu_n\ldots\int g_{m-1} d\mu_{m-1}=
	\left(	\sum_{b\in \mathcal A}	\frac{\mu_{m-1}(b)}{g_n(b)\ldots g_m(b)}\right)^{-1}\int g_{m-1}d\mu_{m-1}.
	$$
	Using (\ref{eq:conformal3})
	$$
	g_{m-1}(c)\mu_{m-1}(c)
	=
	\mu_{m-2}(c)
	\frac{\mu_{m-1}(b)g_{m-1}(b)}{\mu_{m-2}(b)}
	\quad \forall\ b\in \mathcal{A}
	$$
	and summing over $c$  it follows that
	$$
	\int g_{m-1} d\mu_{m-1}
	=
	\frac{\mu_{m-1}(b)g_{m-1}(b)}{\mu_{m-2}(b)},
	$$
	so the following identity holds
	\begin{align*}
	\int g_nd\mu_n\ldots\int g_{m-1}d\mu_{m-1}
	&=
	\left(
	\sum_{b\in \mathcal A}
	\frac{\mu_{m-1}(b)}{g_n(b)\ldots g_m(b)}
	\frac {\mu_{m-2}(b)}{\mu_{m-1}(b)g_{m-1}(b)}
	\right)^{-1}
	\\
	&=
	\left(
	\sum_{b\in A}
	\frac{\mu_{m-2}(b)}{g_n(b)\ldots g_{m-1}(b)}
	\right)^{-1}.
	\end{align*}
	Taking $m=3$ it follows that
	$$ \int g_nd\mu_n...\int g_2 dm_2= \left(\sum_{b\in \mathcal A} \frac{\mu_1(b)}{g_n(b)...g_2(b)}\right)^{-1}.$$
	Since by (\ref{eq:conformal2})
	$$ \sum_{c\in \mathcal A} g_1(c)^{-1} = \frac 1{\mu_1(b) g_1(b)}$$
	for every $b\in \mathcal A$ we obtain
	$$\sum_{c\in\mathcal A} \frac 1{g_1(c)}\sum_{b\in \mathcal A} \frac{\mu_1(b)}{g_n(b)...g_2(b)}= \sum_{b\in \mathcal A} \frac 1{g_n(b)...g_1(b)},$$
	and therefore the claim follows by taking $n\to \infty$.
\end{proof}

\subsection{Uniqueness}

Uniqueness of conformal measures requires a stronger hypothesis. We prove
\begin{theorem}\label{theo:uniqueness}
	Let  the alphabet $\mathcal A$ be finite and suppose that $g_i:X\to \mathbb R_+$ ($i\ge 0,\ g_0$ a constant) satisfy
	\begin{equation}\label{eq:conformal4}
	\sum_{i=0}^\infty \sum_{k=i}^\infty \log \max\{\|g_k\|_\infty,\|g_k^{-1}\|_\infty\}<\infty.
	\end{equation}
	Then there exists exactly one conformal measure for the product type function
	$g(x) =g_0 \prod_{i=1}^\infty g_i(x_i)$. Moreover, this measure is ergodic.
\end{theorem}

\begin{proof}
	Let
	$$ K_i=\prod_{k=i}^\infty \max\{\|g_k\|_\infty^2, \|g_k^{-1}\|_\infty^2\}, \quad i\ge 2 .$$
	Since (\ref{eq:conformal4}) implies the existence condition for a conformal measure of product type, Theorem  \ref{theo:existence},
	guarantees  a  conformal measure for $g$ which is of product type. Denote it by $\mu$ and assume there is another conformal measure $\nu$.
	
	We claim that both measures are equivalent, provided $\frac{\nu([a])}{\mu([a])} \in [K_1^{-1}, K_1]$. In order to show this by induction, assume that for fixed $n\in \mathbb N$ and all cylinder sets $[a_1,\ldots,a_n]$
	$$ \prod_{i=1}^{n+1} K_i^{-1} \le \frac{\mu([a_1,\ldots,a_n])}{\nu([a_1,\ldots,a_n])}\le \prod_{i=1}^{n+1} K_i.$$
	Then for any cylinder $[a_1,\ldots,a_{n+1}]$ we have that
	$$ T([a_1,\ldots,a_{n+1}])= [a_2,\ldots,a_{n+1}]$$
	and hence
	\begin{align*}
	\nu([a_2,\ldots,a_{n+1}])
	&=
	\lambda  \int_{[a_1,\ldots,a_{n+1}]}\prod_{i=1}^\infty g_i(x_i) \nu(dx)
	\\
	&=
	\lambda \prod_{i=1}^{n+1} g_i(a_i)
	\int_{[a_1,\ldots,a_{n+1}]}\prod_{i=n+2}^\infty g_i(x_i) \nu(dx).
	\end{align*}
	The analogue equality holds replacing $\nu$ by $\mu$ and hence
	\begin{align*}
	\frac {\mu([a_1,\ldots,a_{n}])}{\nu([a_1,\ldots,a_{n}])}
	&=
	\frac{\int_{[a_1,\ldots,a_{n+1}]}\prod_{i=n+2}^\infty g_i(x_i) \nu(dx)}
	{\int_{[a_1,\ldots,a_{n+1}]}\prod_{i=n+2}^\infty g_i(x_i) \mu(dx)}
	\\
	&\le
	K_{n+2}\frac{\mu([a_1,\ldots,a_{n+1}])}
	{\nu([a_1,\ldots,a_{n+1}])},
	\end{align*}
	and a similar lower estimate holds interchanging $\mu$ abd $\nu$. This shows that
	\begin{eqnarray*}
		\prod_{i=1}^{n+2} K_i^{-1}&\le& K_{n+2}^{-1}\frac{\mu[a_1,\ldots,a_{n}])}{\nu([a_1,\ldots,a_{n}])} \le \frac{\mu[a_1,\ldots,a_{n+1}])}{\nu([a_1,\ldots,a_{n+1}])} \\
		& \le& K_{n+2} \frac{\mu[a_1,\ldots,a_{n}])}{\nu([a_1,\ldots,a_{n}])}\le \prod_{i=1}^{n+2} K_i.
	\end{eqnarray*}
	Since $K=\prod_{i=1}^\infty K_i <\infty$, the claim is proved.\\
	
	Next we show that a conformal measure $\nu$ satisfies $\nu([a])>0$ for each $a\in \mathcal A$.
	Indeed, let $b\in \mathcal A$ with $\nu([b])>0$. Then for any $a\in \mathcal A$
	$$  \nu([b]) = \nu(T[ab])= \int_{[ab]}  g(x) \mu(dx)$$
	and hence $\nu([a])\ge \nu([ab])>0$ since $g$ does not vanish.
	
	It follows that any two conformal measures are equivalent since $\mathcal A$ is finite.\\
	
	Next we claim that every conformal measure $\nu$ is ergodic: if $A\in \mathcal B$ satisfies $T^{-1}(A)=A$  and $\nu(A)>0$, then it is easy to  see that $\nu(\cdot\cap A)/\mu(A)$ is a conformal measure as well. Then $T^{-1}(A^c)=A^c$ and so $\nu(\cdot\cap A^c)/\nu(A^c)$ is conformal if $\nu(A)<1$. Both measures are singular, contradicting what has been shown so far. Hence $\nu(A)=1$ and $\nu$ is ergodic.\\
	
	Assume now there is another conformal measure $\nu$ which by the previous steps has to be absolutely continuous with  respect to $\mu$. Then there is a function $h>0$, such that,
	$d\nu = h \cdot d\mu$ by the Radon-Nikodym theorem. Since
	\begin{eqnarray*}
		\nu([a_1,\ldots,a_n])&=& \int_{[a_1,\ldots,a_n]} h(x) \mu(dx)\\
		&=&  \lambda \int_{[a,a_1,\ldots,a_n]} h(T(x)) g(x) \mu(dx)\\
		&=& \lambda \int_{[a,a_1,\ldots,a_n]} \frac {h(T(x))}{h(x)} g(x) \nu(dx)
	\end{eqnarray*}
	and
	$$ \nu([a_1,\ldots,a_n])= \lambda \int_{[a,a_1,\ldots,a_n]} g(x) \nu(dx)$$
	we obtain, letting $n\to\infty$ that $\nu$ a.s. $h(T(x))=h(x)$. Now, for every interval $I$ the set
	$A(I)=\{ x\in X: h(x)\in I\}$ is invariant. For each $\eta>0$ there is one interval $I$ of length $\eta$ which has positive measure, hence the conditional measure of $\nu$ restricted to this set $A(I)$ is conformal, and so $\nu(A(I))=1$. Letting the interval shrink to a point $c$ through a sequence of intervals $A(I)$ of measure 1, we see that $h=c $ is constant a.s., finally this implies $c=1$ and $\nu=\mu$.
\end{proof}

\begin{corollary} In case the alphabet is infinite then there is only one conformal measure with
	$$ 0<\inf_{a\in \mathbb N} \frac{\mu([a])}{\nu([a])} $$
	where $\mu$ is the unique conformal measure of product type.
\end{corollary}

\begin{proof}
	In this case the previous proof shows that $\nu$ is absolutely continuous with respect to $\mu$.
\end{proof}

\section{Eigenfunctions of product type}\label{sec:Eigenfunctions of product type}

We now analyse the (point) spectrum of the action of the Ruelle operator on functions of product type. In order to do so, we extend previous definitions to functions $g:X\to \mathbb R_+$ of product type. We say that a measurable function $g: X \to \mathbb R_+ $ of product type
is \emph{$\ell_1$-bounded} if
\begin{equation}\label{eq:strongly_regular}
\sum_{k=2}^\infty \|\log  g_k\|_\infty < \infty,
\end{equation}
and  remark that this condition implies that  $\log g$ is absolutely convergent.
Moreover,  $g$ is called \emph{summable} if $\sum_{a \in \mathcal{A}} g_1(a)< \infty$.
Observe that $g$ is always summable if $\mathcal{A}$ is finite, and that $\ell_1$-boundedness in combination with summability implies that $\|\mathcal{L}_{\log g}(1)\|_\infty < \infty$.

Furthermore, we use  balanced forms for functions $h$ of product type, which are defined by  $h((x_i)_{i\in \mathbb N}) = h_0 \prod h_i(x_i)$  where $h_0 >0$ and $\prod_{a \in \mathcal{A}} h_i(a)=1$ for all $i \in \mathbb N$.
In particular, if $\mathcal{A}$ is finite and $h$ is $\ell_1$-bounded, then $h$ always can be written in balanced form.
Moreover, for a function $g=\prod_{n=1}^\infty g_n$ in balanced form, it follows that $\|\log  g_n\|_\infty \leq v_n(\log g) \leq 2 \|\log  g_n\|_\infty$ for all $n \in \mathbb N$. Hence, Bowen's condition for $\log g$ with index $2$ is equivalent to
\begin{equation} \label{eq:Bowen's condition} \sum_{m=2}^\infty \sum_{n=m}^\infty \|\log g_n\|_\infty < \infty.\end{equation}
Recall from \cite{Wal01} that Bowen's condition has a variety of important consequences when $X$ is compact, like e. g. uniqueness of the equilibrium state, the conformal measure and the eigenfunction of the Ruelle operator. Therefore, the main novelty of the following result is the fact that it is possible to explicitly determine the eigenfunction and the maximal eigenvalue.
 We remark that the eigenvalue coincides with the one from Corollary \ref{cor:lambda} for the $1/g$-conformal measure, even though the construction below relies on the hypothesis that  
 $g$ is in balanced form.

\begin{theorem} \label{theo:eigenfunction-abstract}  Let $(X,T)$ be a topological Bernoulli shift over a finite or countable alphabet $\mathcal A$ and $g$ a  function in balanced form. Then, the Ruelle operator $\mathcal{L} = \mathcal{L}_{\log g}$ maps a   balanced function $h = \prod h_k$ with $|\sum_a g_1(a) h_1(a)|<\infty$ to a balanced function.
	\begin{enumerate}
		\item If $\mathcal{L}(h)=\lambda h$, for $h=\prod h_k$ in balanced form and some $\lambda >0$, then
		\[\lambda = g_0 \sum_{a \in \mathcal{A}} \prod_{k=1}^\infty g_k(a), \quad h_i(a) = \prod_{k > i} g_k(a) \; \forall i\in \mathbb N, a \in \mathcal A.\]
		\item If $g$ satisfies Bowen's condition (\ref{eq:Bowen's condition}) of index 2, then
		the function $h(x)=\prod_{i=1}^\infty h_i(x_i)$, with $h_i$ as above, is defined for all $x \in X$. Furthermore, if $g$ is summable, then $\lambda < \infty$.
	\end{enumerate}
\end{theorem}

\begin{proof} We first show how the Ruelle operator $\mathcal{L} = \mathcal{L}_{\log g}$ acts on the set of  balanced functions. In order to do so, observe that if $h=1$ and $h$ is in balanced form, then all the entries of $h$ have to be equal to one. In particular, there exists at most one balanced form of a function.
	For $h$  in balanced form, we have
	\begin{equation}\label{eq:auto-function}
	\mathcal{L}(h)(x) = \sum_{a \in \mathcal{A}} g(ax) h(ax)
	= g_0 h_0 \sum_{a \in \mathcal{A}} g_1(a) h_1(a) \prod_{i=1}^\infty g_{i+1}(x_i) h_{i+1}(x_i).
	\end{equation}
	Hence, provided that $\sum_a g_1(a) h_1(a)$ is finite,
	the balanced form of $\mathcal{L}(h)$ is given by $(\mathcal{L}(h))_0 = g_0 h_0 \sum_{a \in \mathcal{A}} g_1(a) h_1(a)$ and $(\mathcal{L}(h))_i  =  g_{i+1} h_{i+1}$ for all $i\in \mathbb N$.
	
	\noindent\emph{Proof of}\! (i). Assume that
	$\mathcal{L}(h)=\lambda h$, for $h$ in balanced form with $h_0 =1$.
	It follows from (\ref{eq:auto-function}) that
	$\mathcal{L}(h)=\lambda h$ implies that
	$\lambda= g_0 \sum_{a \in \mathcal{A}} g_1(a) h_1(a)$
	and $h_i =   g_{i+1} h_{i+1}$ for all $i\in \mathbb N$.
	Hence, by induction,
	\[
	h_i
	=
	\prod_{k > i} g_k\; \ \ (\forall i\in \mathbb N), \quad \lambda
	=
	g_0 \sum_{a \in \mathcal{A}} \prod_{i=1}^\infty g_k(a).
	\]
	\noindent\emph{Proof of}\! (ii). Bowen's condition implies that $ \sum_{i\geq 2} \log g_i$ is an absolutely convergent series. Hence, $h(x)$ exists for all $x \in X$. In order to show the existence of $\lambda$, note that by summability,
	\begin{align*}
	\lambda
	=
	g_0 \sum_{a \in \mathcal{A}} \prod_{k=1}^\infty g_k(a)
	\leq
	g_0 \left({ \textstyle \sum_{a \in \mathcal{A}}  g_1(a) }\right)
	e^{ \sum_{k=2}^\infty \| \log g_k \|_\infty}
	<
	\infty.
	\end{align*}
	Since $ \prod_{k} g_k(a) >0$, it follows from this that $\lambda$ exists.
\end{proof}

Observe that the theorem does not state that the space of balanced functions of product type is $\mathcal{L}$-invariant due to the fact that the
sum  $\sum_{a \in \mathcal{A}} g_1(a) h_1(a)$ might be not well defined if $\mathcal{A}$ is infinite. In order to construct an invariant function space in this case one has to consider subclasses of potentials and functions of product type. For example, it easily follows from the argument in the first part of the above proof that, if $g=\prod_i g_i$ is summable and $\| g_i \|_\infty< \infty $ for all $i$, then  $\mathcal{L}_{\log g}$ acts on the space
\[ \{    f = {\textstyle \prod_i } f_i :  \| g_i \|_\infty< \infty  \; \forall i =1,2,\ldots\}. \]

The main motivation of this note is to consider potentials
beyond Bowen's condition. In particular, it will turn out
that Bowen's condition is a sharp condition for the existence of a continuous eigenfunction $h$.
However, the situation with respect to measures is somehow satisfactory,
as it is possible to explicitly construct conformal measures and
equilibrium states for $\ell_1$-bounded potentials.
In order to do so, we first have to introduce the
action of $\mathcal{L}_{\log g}$ on measures and the notions
of pressure and equilibrium states.

If $g$ is  $\ell_1$-bounded  and summable, then $\log g$
is locally uniformly continuous and $\|\mathcal{L}_{\log g}(1)\|_\infty < \infty$.
Moreover, by a standard calculation, $\mathcal{L}_{\log g}$
acts on uniformly continuous functions.
In particular,
$
\int h d\mathcal{L}_{\log g}^\ast \mu
=
\int \mathcal{L}_{\log g}(h) d \mu
$, for bounded continuous functions $h$,
defines an operator $\mathcal{L}_{\log g}^*$ on the space of finite signed Borel measures on $X$.

We now recall the definition of the pressure for countable state Markov shifts from
\cite{SU07}.
As it is shown in there, the pressure $P(\log g)$ defined by
\[ P(\log g) := \lim_{n \to \infty} \frac{1}{n} \log \sum_{a \in \mathcal{A}^n} \sup_{x \in [a]} \prod_{i=0}^{n-1} g\circ T^i(x) \]
exists by subadditivity, but is not necessarily finite.
However, as shown below, $P(\log g) < \infty$
for $\ell_1$-bounded, summable potentials $g$.
Also recall that, if $\mathcal A$ is finite and $\log g$
is continuous, the variational principle (\cite{Wal1975})
\[
P(\log g)
=
\sup \{
h_m(T) + {\textstyle \int \log g\ dm} :
m \hbox{ probability with }m=m \circ T^{-1}
\}
\]
holds, with $h_m(T)$ denoting the Kolmogorov-Sinai entropy. Furthermore, if $m$ is an invariant probability measure which realizes the supremum, $m$ is referred to as an \emph{equilibrium state}. However, note that this notion is only applicable if $\mathcal A$ is finite since it is unknown whether a
variational principle holds for general locally bounded, summable potentials.

The construction of an equilibrium state for topological Bernoulli shifts is  based on the following
observation which reveals the independence from the existence of the eigenfunction $h$.
Namely, a formal calculation gives, for $x= (x_i)_{i\in \mathbb N}$, that
\begin{eqnarray}
\nonumber \frac{g(x)\, h(x)}{\lambda \cdot  h\circ T(x)}
&=& \frac{g(x)}{\lambda} \prod_{i=1}^\infty \frac{h_i(x_i)}{h_i(x_{i+1})}
=
\frac{g(x) h_1 (x_1)}{\lambda}
\prod_{i=1}^\infty \frac{h_{i+1}(x_{i+1})}{h_i(x_{i+1})}
\\
\label{def:normalized_function}
&=&
\frac{g(x) h_{1}(x_1)}{\lambda \prod_{i=1}^\infty  g_{i+1}(x_{i+1})}
=
\frac{\prod_{i=1}^\infty  g_i(x_1)}{\sum_{a \in \mathcal{A}}
	\prod_{i=1}^\infty  g_i(a) }
=:
\tilde g(x).
\end{eqnarray}
Hence, even though the function $h$ might not exist, the quotients $h/h\circ T$
and $\tilde g = gh/(\lambda h \circ T)$ are well defined for summable,
locally bounded $g$.

The following theorem now provides explicit constructions of
conformal measures and equilibrium states as well as a partial
answer to the existence of the eigenfunction.
If the sequence $(\log h_i)_{i\in\mathbb{N}}$, with $(h_i)$
as above is square summable, then the eigenfunction $h$
exists a.e. with respect to the Bernoulli measure of maximal entropy,
but not necessarily with respect to the conformal measure (see the class of examples in Section \ref{exa}).

The motivation for the following definition, equivalent to \eqref{eq:intermediate} above, is to provide a sufficient condition for this property. We say that 
 $g$ has \emph{$\ell_2$-bounded tails} if there exists $k \in \mathbb N$ such that
 \begin{equation} \label{eq:square-summable}
 \sum_{i=k}^\infty \sup_{a \in \mathcal{A}} \bigg({ \sum_{j=i}^\infty }\log g_j(a)\bigg)^2< \infty,
\end{equation}

\begin{theorem}
	\label{theo:equilibrium}
	Let $(X,T)$ be a topological Bernoulli shift over a finite
	or countable alphabet $\mathcal A$ and  let $g$ be a $\ell_1$-bounded, summable potential function.
	Furthermore, let $\lambda$ be as in 
	Theorem \ref{theo:eigenfunction-abstract} and assume that $\mu=\otimes_{n=1}^\infty \mu_n$ is a measure of product type and $\tilde \mu$ is the Bernoulli measure with weights $\{ \tilde\mu_0(a) : a\in \mathcal A\}$, where
	\begin{equation}
	\nonumber
	\label{eq:conformal+equilibrium}
	{\mu}_n(a) := \prod_{i=1}^n g_i(a) \left/
	\sum_{b \in \mathcal{A}} \prod_{i=1}^n g_i(b) \right. ,
	\quad
	{\tilde \mu}_0(a) := \prod_{i=1}^\infty g_i(a)
	\left/ \sum_{b \in \mathcal{A}} \prod_{i=1}^\infty g_i(b) \right.
	.
	\end{equation}
	\begin{enumerate}
		\item We have $\mathcal{L}_{\log g}^\ast(\mu)=\lambda \mu$, $\mathcal{L}_{\log \tilde g}^\ast(\tilde \mu)=\tilde \mu$, $\log \lambda = P(\log g)$  and
		\[ P(\log g) = h_{\tilde \mu}(T) + {\textstyle \int \log g d\tilde \mu}.\]
		If $\mathcal{A}$ is finite, then $\tilde \mu$ is an equilibrium state.
		\item If $g$ is in balanced form, $\mathcal{A}$ is finite and, for some $k>1$,  \eqref{eq:square-summable} holds,
		then $h(x)$ defined as in Theorem \ref{theo:eigenfunction-abstract} exists for almost every $x \in X$ with respect to the $(1/|\mathcal{A}|, \ldots, 1/|\mathcal{A}|)$-Bernoulli measure on $X$. Furthermore,   $\mathcal{L}_{\log g}(h)=\lambda h$.
	\end{enumerate}
\end{theorem}

\begin{proof} As it is well known, $\mathcal{L}_{\log g}^\ast(\mu)=\lambda \mu$ if and only if $\mu$ is $1/g$-conformal. Hence, by  the first part of Theorem \ref{theo:existence}, we have that $\mu$ is given by $\mu_n$ as in the statement of the theorem. In order to verify that $\lambda$ is as in  Theorem \ref{theo:eigenfunction-abstract}, note that by bounded convergence,
	\begin{eqnarray*}
		\int \mathcal{L}_{\log g} 1d\mu &=& g_0 \sum_{b \in \mathcal{A}} g_1(b) \prod_{i=1}^\infty \int g_{i+1}(x_i) d\mu_i(x_i) \\
		&= & g_0 \sum_{b \in \mathcal{A}} g_1(b)  \prod_{i=1}^\infty  \frac{\sum_{a \in \mathcal{A}} g_1(a) \cdots g_{i+1}(a)}{\sum_{a \in \mathcal{A}} g_1(a) \cdots g_{i}(a)} \\
		&=& g_0 \lim_{i \to \infty} \sum_{a \in \mathcal{A}} g_1(a) \cdots g_{i+1}(a) = g_0 \sum_{a \in \mathcal{A}} \prod_{i=1}^\infty   g_i(a).
	\end{eqnarray*}
	Hence,  $\mathcal{L}_{\log g}^\ast(\mu)=\lambda \mu$ with $\lambda$ as in Theorem \ref{theo:eigenfunction-abstract}.  In order to show that $\mathcal{L}_{\log \tilde g}^\ast(\tilde \mu)=\tilde \mu$, note that
	$\tilde g$ as defined in  (\ref{def:normalized_function}) only depends on the first coordinate and in particular is of product type and locally bounded. Furthermore, it follows from
	$\mathcal{L}_{\log \tilde g} (1) =1$ that $\tilde g$ is summable. Hence, $\mathcal{L}_{\log \tilde g}^\ast(\tilde \mu)=\tilde \mu$ again by the first part of Theorem \ref{theo:existence}.
	
	We now establish $P(\log g) = h_{\tilde \mu}(T) + {\textstyle \int \log g\  d\tilde \mu}$ by proving that $h_{\tilde \mu}(T) = \log \lambda - \int \log g d\tilde \mu$ and $P(\log g)=\log \lambda$. As $\tilde \mu$ is a Bernoulli measure we obtain
	\begin{align*}
	h_{\tilde \mu}(T) &=  - \int  \log \tilde \mu([x_1])  \tilde \mu (d(x))\\
	&= -\sum_{a \in  \mathcal{A}} \tilde \mu_0(a)
	\left( \log \prod_{i=1}^\infty g_i(a) - \log \sum_{b \in \mathcal{A}} \prod_{i=1}^\infty g_i(b) \right) \\
	&= \log \sum_{b \in \mathcal{A}} \prod_{i=1}^\infty g_i(b) - \sum_{i=1}^\infty \sum_{a \in  \mathcal{A}}   \log g_i(a) \tilde \mu_0(a)
	= \log \lambda - \int \log g d \tilde \mu.
	\end{align*}
	In order to show that $P(\log g)=\log \lambda$, note that $\ell_1$-boundedness implies for $x,y \in [a_1, \ldots,a_n]$ that there exists $C>0 $ such that
	\[\log  \prod_{k=0}^{n-1} \frac{g(T^k(x))}{g(T^k(y))}  \leq  2 \sum_{k=0}^{n-1} \sum_{i \geq k} \|\log g_i\|_\infty \leq Cn. \]
	Hence, $\mathcal{L}_{\log g}^n(1)(x) = e^{\pm Cn} \mathcal{L}_{\log g}^n(1)(y)$ for all $x,y \in X$. Since $\log n/n \to 0$, we have
	\begin{align*}
	P(\log g)
	&= \lim_{n\to \infty} \frac{1}{n} \log \mathcal{L}_{\log g}^n(1)(x)  = \lim_{n\to \infty} \frac{1}{n} \int \log \mathcal{L}_{\log g}^n(1) d\mu \\
	&= \lim_{n\to \infty} \frac{1}{n} \log  \int 1 d (\mathcal{L}_{\log g}^n)^\ast (\mu)
	= \log \lambda.
	\end{align*}
	Hence, assertion 1 is proven. In order to show assertion 2, let
	$\rho$ denote the $(1/|\mathcal{A}|, \ldots, 1/|\mathcal{A}|)$-Bernoulli measure on $X$, the measure of maximal entropy. Write $\rho = \otimes \rho_i$, the product of the equidistribution $\rho_i$ on $\mathcal A$. With respect to this measure, and since $g$ is balanced it follows that, for all   $j \geq k$,
	\[
	\int \log h_j  d\rho = \int \log h_j(a)  d\rho_j(a) =
	\sum_{i>j } \int  \log g_i(a) d\rho_j(a) = 0.
	\]
	We now consider $(h_i)_{i\in \mathbb N}$ as a stochastic processes on the probability space $(X, \rho)$.  In particular, the above implies that $\mathbb E(\log h_j)=0$. Furthermore, for the variances of $\log h_j$, we obtain
	\[ \hbox{Var}(\log h_j) = \int (\log h_j)^2 d\rho \leq \max_{a \in \mathcal{A}} (\log h_j(a))^2 = \max_{a \in \mathcal{A}} \left({\textstyle \sum_{i>j} }\log g_i(a)\right)^2.
	\]
	Hence, the summability condition implies that $\sum_{j>k} \hbox{Var}(\log h_j) < \infty$. 
	As a consequence of Kolmogorov's three series theorem (as in \cite[Corollary 3 on p. 87]{Luk75}), it follows that $\log h = \sum_{j \geq 1} \log h_j$ converges $\rho$-a.s. The remaining assertion  $\mathcal{L}_{\log g}(h)=\lambda h$ follows as in
	Theorem \ref{theo:eigenfunction-abstract}.
\end{proof}

The existence of $h$ in the second part of the above theorem is based on the fact that the log of a balanced function has zero integral with respect to the measure of maximal entropy. By considering a suitable scaling of $h$, an analogous result holds with respect to $\mu$. The existence of this function is equivalent to the equivalence of the measures $\mu$ and $\tilde \mu$.

\begin{theorem}
	\label{theo:invariant-function-unbalanced}
	Let $(X,T)$ be a topological Bernoulli shift over a finite or countable alphabet $\mathcal A$, let $g$ be a $\ell_1$-bounded, summable potential 
	function of product type and let $\mu$ and $\lambda$ be as in Theorem \ref{theo:equilibrium}.
\begin{enumerate} 
\item There is at most one $h \in L^1(X,\mu)$ with $ \mathcal{L}_{\log g}(h)
 = \lambda h$ and $\int h d\mu=1$. 
\item If 
\eqref{eq:square-summable} holds for some $k \in \mathbb N$, then
 the function 
 \begin{equation}\label{eq:non-normalized-eigenfunction}
 h_\mu ((x_j)) = \prod_{j=1}^{\infty} \frac{\sum_{a \in \mathcal{A}} \prod_{l=1}^j g_l(a) }{\sum_{a \in \mathcal{A}} \prod_{l=1}^\infty g_l(a)} \prod_{l=1}^\infty g_{l+j}(x_j)
\end{equation}
 is in $L^1(X,\mu)$. Furthermore, $\int h_\mu d\mu=1$,  $ \mathcal{L}_{\log g}(h_\mu)
 = \lambda h_\mu $ and $d\tilde{\mu} = h_\mu d\mu $.
 \item The function $h_\mu$ exists $\mu$-a.s.. 
 Moreover, $\int h_\mu d\mu>0$ if and only if $\mu$ and $\tilde{\mu}$ are equivalent. 
 If $\int h_\mu d\mu=0$, then $\tilde{\mu}$ and $\mu$ are singular measures.
 \end{enumerate}
\end{theorem}

\begin{proof} 
\noindent  (i) In order to show uniqueness, we will identify $\lambda^{-1}\mathcal{L}_{\log g}$ with the transfer operator. As it was noted above, $\ell_1$-boundedness and summability imply that $\lambda^{-1}\mathcal{L}_{\log g}$ acts on uniformly continuous functions. It now follows from the conformality of $\mu$ that $\lambda^{-1}\mathcal{L}_{\log g}$ acts as the transfer operator on $L^1(X,\mu)$, that is $\int \psi \lambda^{-1}\mathcal{L}_{\log g}(\phi) d \mu  =  \int \psi\circ T \cdot \phi \, d \mu$ for all $\psi \in L^\infty(X,\mu)$ and $\phi \in L^1(X,\mu)$. A further important ingredient is exactness, that is triviality of the tail $\sigma$-field $\bigcap_{n>1} T^{-n} \mathcal{B}$ modulo $\mu$. As $\mu$ is a product measure, it follows from Kolmogorov's 0-1 law that $T$ is exact. Hence, by Lin's criterion for exactness (\cite{Lin1971}, Th. 4.4)
	\[ \lim_{n \to \infty} \|  \lambda^{-n}\mathcal{L}_{\log g}^n(\phi)\|_1 =0 \]
	for all $\phi \in L^1(X,\mu)$ with $\int  \phi\, d\mu=0$.	
	In particular, if $\mathcal{L}_{\log g}(h) = \lambda \phi$ and $\int  h\, d\mu=0$, then $ \|  h \|_1 =0$. Hence, if $h_1,h_2$ satisfy  $\mathcal{L}_{\log g}(h_i) = \lambda h_i$ and $\int h_i d\mu=1$, then $ \|  h_1-h_2 \|_1 =0$. This proves the uniqueness of $h$.
	
\noindent (ii) In order to show that $h_\mu$ exists, we employ Kolmogorov's three series theorem as in \cite[Corollary 1 on p. 84]{Luk75}. Hence we have to show that $|\sum  \int \log h_\mu^{(j)} d\mu_j|<\infty$ and $\sum \int (\log h_\mu^{(j)})^2 d\mu_j < \infty$, for 
\[h_\mu^{(j)} := \Delta_j \prod_{l=1}^\infty g_{l+j}(x_j), \hbox{ where } \Delta_j := \frac{\sum_{a \in \mathcal{A}} \prod_{l=1}^j g_l(a) }{\sum_{a \in \mathcal{A}} \prod_{l=1}^\infty g_l(a)}.\]
By   construction of $\mu$, we have  $\int  h_\mu^{(j)} d\mu_j =1$ and, by Jensen's inequality, $\int \log  h_\mu^{(j)} d\mu_j \leq 0$. In order to prove summability of the first sum, it therefore suffices to obtain a lower bound which follows from
 \begin{align*}
\int \log  h_\mu^{(j)} d\mu_j &= \int \sum_{l>j} \log g_ld\mu_j - \log \frac{1}{\Delta_j}
 \geq 
\int \sum_{l>j} \log g_l d\mu_j  + 1 - \frac{1}{\Delta_j} \\
&= 
\int  \sum_{l>j} \log g_l d\mu_j  +  
\frac{{\sum_{a \in \mathcal{A}} \prod_{l=1}^j g_l(a) (1 - \prod_{l>j} g_l(a))}}{{\sum_{a \in \mathcal{A}} \prod_{l=1}^j g_l(a) }}\\
&=
\int \sum_{l>j} \log g_l +  1-  \prod_{l>j} g_l \; d\mu_j = o\left( \sup_a (1- \textstyle \prod_{l>j} g_l(a))^2  \right)
 \end{align*}
where we used $\log (1+x) -x = o(x^2)$ in the last identity. Hence, if \eqref{eq:square-summable} holds, then 
$\sum \int \log h_\mu^{(j)} d\mu_j$ is summable. 
Using a similar argument, it easily can be seen that $\log \Delta_j \sim \int \sum_{l>j} \log g_l d\mu_j$. Hence, if \eqref{eq:square-summable} holds, then 
$\sum \int (\log h_\mu^{(j)})^2 d\mu_j$ is summable. In particular, $h_\mu$ exists $\mu$-a.s. by the three series theorem whereas it follows from $\int  h_\mu^{(j)} d\mu_j =1$ that $\int h_\mu d\mu =1$.

In order to show that $d\tilde{\mu} = h_\mu d\mu$ it suffices to show that $\tilde{\mu}([w]) = \int_{[w]} h_\mu d\mu$, for each $n \in \mathbb N$ and $w=(w_1,\ldots, w_n)$ with $w_j \in \mathcal{A}$. It follows from the product structure that 
\begin{align*} \int_{[w]} h_\mu d\mu  = \prod_{j=1}^n  \int_{[w_j]} h_\mu^{(j)} d\mu_j = \prod_{j=1}^n \Delta_j  \frac{\prod_{l=1}^j g_{l}(w_j)}{\sum_a \prod_{l=1}^j g_{l}(a)} \prod_{l=1}^\infty g_{l+j}(w_j)  = \tilde{\mu}([w]).
\end{align*} 
Hence, $h_\mu = d\tilde{\mu}([w])/d\mu$. As $\lambda^{-1}\mathcal{L}_{\log g}$ acts as the transfer operator and $d\tilde{\mu} = h_\mu d\mu$ is invariant, it follows for each test function $\phi \in L^\infty(X,\mu)$, that
	\[ \int \phi \, \lambda^{-1}\mathcal{L}_f(h_\mu) d\mu = \int \phi \circ T \cdot  h_m d\mu = \int\phi \, h_m d\mu.\]
Hence, $\mathcal{L}_f(h_\mu) = \lambda h_\mu$.

\noindent (iii) In order to prove the third part of the theorem, we will make use of the fact, that the shift space $X$ is a Besicovitch space and therefore, a measure differentiation theorem holds (see \cite{Bogachev:2007}). That is, the function
\[ D_\mu(\tilde{\mu}) ((x_j)) = \lim_{n \to \infty} \frac{\tilde{\mu}([x_1, \ldots , x_n])}{{\mu}([x_1, \ldots , x_n])}\]
exists and is finite $\mu$-a.e.. Moreover, $D_\mu(\tilde{\mu})$ is the Radon-Nikodym derivative $d\tilde{\mu}_{\hbox{\tiny ac}} /d\mu$, where $\tilde{\mu}_{\hbox{\tiny ac}}$ is the absolutely continuous part of $\tilde{\mu}$ with respect to $\mu$.   

In order to apply the result, 
observe that $D_\mu(\tilde{\mu}) =0$ implies that $\tilde{\mu}$ 
and $\mu$ are singular measures. 
However, if $\int D_\mu(\tilde{\mu}) d\mu = \tilde{\mu}_{\hbox{\tiny ac}}(X)>0$, 
it follows from ergodicity of $\tilde \mu$ that $\tilde{\mu}_{\hbox{\tiny ac}} =\tilde{\mu}$ 
and from 
\[ \frac{\tilde{\mu}([x_1, \ldots , x_n])}{{\mu}([x_1, \ldots , x_n])} = \prod_{j=1}^n \frac{\prod_{l=1}^\infty g_l(x_j) /(\sum_a \prod_{l=1}^\infty g_l(a))}{\prod_{l=1}^j g_l(x_j) /(\sum_a \prod_{l=1}^j g_l(a))} =  \prod_{j=1}^n \Delta_j \prod_{l=1}^\infty g_{l+j}(x_j) \]
that  $ D_\mu(\tilde{\mu})$ and $h_\mu$ are equal $\mu$-a.s.. 
It follows from ergodicity of $\mu$ that $ D_\mu(\tilde{\mu})> 0$ a.s.   
\end{proof} 


\section{Eigenfunctions in $L^1$-spaces}\label{sec:L1-spaces}

The Ruelle operator $\mathcal L_f$ with $f\in C(X)$ acts on classes of measurable functions modulo any Bernoulli measure $\rho$ on $X$ of the form
$\rho=\otimes_{i=1}^\infty \rho_0$,
where $\rho_0$ is any probability measure on $\mathcal A$. Indeed, note that $\rho$ is a shift invariant and ergodic measure on $(X,T)$.  Let $\phi,\psi$ be two functions which agree $\rho$ almost surely. Let $A=\{ \phi=\psi\}$. Then $\rho(A)=1$  and because of invariance of $\rho$ we may assume that $T^{-1}(A)\subset A$. Then by definition $\mathcal L_f\phi=\mathcal L_f\psi$ on $A$ (if the operator is well defined for these functions), so that $\mathcal L_f$ maps equivalence classes of measurable functions into such classes.

In this section we always assume that the alphabet $\mathcal A$ is finite. Then the Ruelle operator is always well defined on measurable functions. When the Ruelle operator is well defined in case of an infinite alphabet the following results can be adapted. The first theorem is a slightly modified and extended result from Theorem  \ref{theo:equilibrium}, part 2.

\begin{theorem}\label{l1theorem} \label{theo:Lp-existence}Let $g=e^f= \prod_{i=0}^\infty g_i$ be a balanced potential function.
	
	1. The Ruelle operator $\mathcal{L}_{f}$ defines canonically a
	bounded linear operator on
	$L^{p}(X,\rho)$ for all $1\leqslant p\le \infty$, where
	$\rho=\otimes_{i=1}^\infty \rho_0$ is any stationary Bernoulli measure.
	
	2. Assume that $g$ has $\ell_2$-summable tails, that is for some $k>1$,   
	$$M:=\sum_{i=k}^\infty \max_{a \in \mathcal{A}} \bigg({ \sum_{j=i}^\infty }\log g_j(a)\bigg)^2< \infty, $$
	and that $\rho$ is a stationary  Bernoulli measure with
	$$ \int \log g_k(x) \rho(dx)=0\qquad \forall k\ge 1.$$
	Then the function $h:X\to \mathbb R_+\cup\{\infty\}$ defined by
	$$ h((x_i)_{i\in \mathbb N})=  \prod_{i=1}^\infty h_i(x_i)\qquad h_i(a)=\prod_{k>i}g_k(a)\ \ a\in \mathcal A$$
	belongs to $L^p(X,\rho)$ for every $1\le p<\infty$ and is an
	almost surely positive eigenfunction of $\mathcal L_f:L^p(X,\rho)\to L^p(X,\rho)$  with eigenvalue
	$$  \lambda = g_0\sum_{a\in \mathcal A} \prod_{k=1}^\infty g_k(a).$$
\end{theorem}

\begin{proof} 1.
	We need to show that $\mathcal{L}_{f}$ sends
	$L^p(X,\rho)$ into itself.
	Indeed, let $1\leqslant p<\infty$ be fixed and
	$\varphi\in L^p(X,\rho)$.
	Bounding $f$ from above  by its supremum norm and using the triangular
	inequality we get
	\begin{align*}
	|\mathcal{L}_{f}(\varphi)(x)|^{p} &=|\sum_{a\in \mathcal A} \varphi(ax) g(ax)|^p
	\le \|g\|_{\infty}^p \sum_{a\in \mathcal A}|\varphi(ax)|^p.
	\end{align*}
	By the hypothesis
	\[
	\int_{X} |\varphi(x)|^p d\rho(x)<+\infty
	\]
	and since $\rho$ is a Bernoulli measure,
	$$\sum_{a\in \mathcal A} \int_X |\phi(ax)|^p \rho(dx)=|\mathcal A|  \int_X |\phi(x)|^p \rho(dx)<\infty,$$
	thus proving that $\mathcal{L}_{f}$ sends
	$L^p(X,\rho)$ to itself in
	case $1\le p<\infty$.
	The case $p=\infty$ is trivial
	because the Ruelle operator is just a finite sum of
	a product of two uniformly bounded functions.
	
	This estimate also shows that $\mathcal L_f$ can be considered as a  bounded operator acting on $L^p(X,\rho)$ for $1\le p\le \infty$.
	
	2. We first show that $h$ is almost surely finite. Similarly to the proof of Theorem  \ref{theo:equilibrium} the random variables $\log h_j$ satisfy $\int \log h_j d\rho=0$ and
	$$ \mbox{Var}(\log h_j) \le \max_{a\in \mathcal A} \left(\sum_{i>j} \log g_i(a)\right)^2.$$
	Again by Kolmogorov's three series theorem  (\cite[p. 87]{Luk75})
	$ \sum_{i=1}^\infty \log h_i$
	converges $\rho$ a.s..
	
	We show next that the moment generating function for $H=\sum_{n=1}^\infty  \log h_n$ exists on $\mathbb R$.
	Since $\log h_n(x)\le \max_{a\in \mathcal A} \sum_{i>n} g_i(a)$ it follows from independence of $\log h_n$ that    that for $p\ge 2$
	$$  E|H|^p \le  (M^{2p-2})^{1/2} (EH^2)^{1/2} \le M^{p-1}\sum_{n=1}^\infty E(\log h_n)^2 \le M^p$$
	whence
	$$ Ee^{tH}= \sum_{n=0}^\infty  \frac{t^n }{n!} EH^n\le \sum_{n=0}^\infty \frac{(tM)^n}{n!}<\infty.$$
	In particular, for $p\in \mathbb N$
	$$ Eh^p= Ee^{pH}<\infty$$
	and $h\in L^p(X,\rho)$.
	
	The proof is completed  similar
	to the one given in Theorem \ref{theo:eigenfunction-abstract}.
\end{proof}

We finally turn towards uniqueness questions of the eigenfunction $h$. We assume that the alphabet $\mathcal A$ is finite.

The uniqueness of the eigenfunction $h$ with respect to the eigenvalue $\lambda$ takes the following form.
Recall that
$$ \mu_0(a) =\lambda^{-1} \prod_{l=1}^\infty g_j(a)\qquad a\in \mathcal A$$
defines the equilibrium product measure. The operator
$$ P_{\mu_0} \psi(x_1,x_2,...)= \sum_{a\in \mathcal A} \psi(a,x_1,x_2,...)\mu_0(a)$$
acts on measurable functions  and on $\rho$-equivalence classes in $L^1(X,\rho)$, whence $P_{\mu_0}$  will be considered as an operator on $L^1(X,\rho)$.

\begin{theorem}\label{theo:uniqueness-of-h}
	Let $\mathcal A$ be a finite alphabet and $\rho$ be a product measure as in the previous theorem and $g=e^f$ be a balanced potential with $\ell_2$-summable tails. 
	Then  the Ruelle operator $\mathcal L_f:L^1(X,\rho)\to L^1(X,\rho)$ has (up to multiplication by constants) exactly one eigenfunction $h\in L^1(X,\rho)$ with respect to the eigenvalue
	$$ \lambda = g_0\sum_{a\in A}\prod_{k=1}^\infty g_k(a)$$
	if and only if $P_{\mu_0}$ is ergodic (i.e. has only one eigenfunction for the eigenvalue $1$ up to multiplication by constants).
\end{theorem}

\begin{proof}
	We only need to show uniqueness. Let $\phi\in L^1(X,\rho)$ be an eigenfunction for the eigenvalue $\lambda$.
	Let $X_1,X_2,...$ denote the i.i.d. coordinate process determining $\rho$.  Then
	\begin{eqnarray*}
		&&\phi(X_1,X_2,...)= \lambda^{-1} \mathcal L_f \phi (X_1,X_2,...)\\
		&& \quad = \lambda^{-1} \sum_{a \in \mathcal A} \phi(a, X_1,X_2,...)\prod_{k=1}^\infty  g_{k+1}(X_k)g_1(a)
	\end{eqnarray*}
	and dividing by $h(X_1,X_2,...)$ yields
	\begin{eqnarray*}
		\frac{\phi(X_1,X_2,...)}{h(X_1,X_2,...)} &=& \lambda^{-1}\sum_{a\in \mathcal A} \frac {\phi(a,X_1,X_2,...)}{h(X_1,X_2,...)} \prod_{k=1}^\infty g_{k+1}(X_k) g_1(a)\\
		&=& \lambda^{-1}\sum_{a\in\mathcal A} \frac{\phi(a,X_1,X_2,...)}{\prod_{k=1}^\infty\prod_{j=k+2}^\infty g_j(X_k)} g_1(a)\\
		&=&  \lambda^{-1}\sum_{a\in\mathcal A} \frac{\phi(a,X_1,X_2,...)}{h(a,X_1,X_2,...)} \prod_{l=1}^\infty g_l(a)\\
		&=& P_{\mu_0} \frac {\phi}{h}(X_1,X_2,...)
	\end{eqnarray*}
	Therefore $\phi/h$ is an eigenfunction for the eigenvalue $1$ (note that $P_{\mu_0}1=1$).
	Thus if $P_{\mu_0}$ is ergodic, $\phi/h$ is constant.
	
	Conversely, the above equation shows that if $P_{\mu_0}$ has another eigenfunction $\psi$, then $\psi h$ is an eigenfunction for $\mathcal L_f$, proving the theorem.
\end{proof}

\section{The leading example} \label{exa}
We return to the class of potentials defined in Example \ref{ex:simplified-dyson}. Recall that it uses the  alphabet $\mathcal{A}=\{-1,1\}$
and potentials of the form
\begin{equation} \label{ii}
f(x)= \sum_{n=1}^{\infty} \frac{x_n}{n^{\gamma}},
\quad \gamma>1.
\end{equation}
Observe that $g(x) := e^{-f(x)}$ satisfies $\inf_x g(x)  = \exp(-\sum_n n^\gamma)>0$, whence the potential of product type $g$ is bounded from below. We also have  that  $g$ is balanced and $\ell_1$-bounded. Hence, we obtain explicit expressions for the conformal measure, the equilibrium state and $\lambda$ by applying Theorems \ref{theo:existence}, \ref{theo:eigenfunction-abstract} and \ref{theo:equilibrium}. In here, $\zeta(\gamma)$ refers to the Riemann $\zeta$-function  $\zeta(\gamma):= \sum_{j=1}^{\infty} j^{-\gamma}$.
\begin{enumerate}
	\item
	The conformal measure
	$\mu=\otimes_{i=1}^\infty \mu_i$
	is of product type, where
	\begin{equation}\label{eq:conformal-measure-example}
	\mu_i(\{1\})
	=
	\frac{\exp(\sum_{j=1}^{i}j^{-\gamma})}{2\cosh(\sum_{j=1}^{i} j^{-\gamma}) },
	\qquad
	\mu_{i}(\{-1\})
	=
	\frac{\exp(-\sum_{j=1}^{i}j^{-\gamma})}{2\cosh(\sum_{j=1}^{i} j^{-\gamma}) }.
	\end{equation}
	
	\item
	The conformality parameter is equal to $\lambda = 2\cosh(\zeta(\gamma))$.
	
	\item
	The equilibrium state
	$\tilde{\mu} = \otimes_{i=1}^{\infty} \tilde{\mu}_i$ is a Bernoulli measure (that is a $ \tilde{\mu}_i =  \tilde{\mu}_j$ for all $i,j$).
	The measure $\tilde{\mu}_0:= \tilde{\mu}_i$ is given by
	\begin{equation}\label{eq:equilibrium-example}
	\tilde{\mu}_0(\{1\})
	=
	\frac{\exp(\zeta(\gamma))}{2\cosh(\zeta(\gamma)) },
	\qquad
	\tilde{\mu}_0(\{-1\})
	=
	\frac{\exp(-\zeta(\gamma))}{2\cosh(\zeta(\gamma)) }.
	\end{equation}
\end{enumerate}

\subsection{Bowen's class ($\gamma>2$)}
Recall that it has been shown above that $f$ is in Bowen's class if and only if $\gamma>2$. In this situation, we obtain a stronger result. Namely, by Theorem \ref{theo:uniqueness}, the  measure $\mu$ above is the unique conformal measure. In particular, $\lambda$ is also uniquely determined by $\mathcal{L}^\ast_{f}(\mu)= \lambda \mu$. Moreover, the
function $h((x_i))=\prod_{i\geq 1} h_i(x_i)$ defined by
\begin{equation} \label{eq:eigenfunction} h_n(x_n) := \exp (\alpha_n x_n), \qquad   \alpha_n := \sum_{j=n+1}^{\infty}j^{-\gamma},
\end{equation}
is an eigenfunction of product type.
This function is the unique function with $\mathcal{L}_{f}(h)= \lambda h$, and the equilibrium state is given by, as usual, $d\tilde{\mu} = h d\mu$.
It is worth noting  that for $\gamma>2$, Walters showed in \cite{Wal01} that a Perron-Frobenius theorem holds in a more general situation. Furthermore, the main result in \cite{BFG99} is applicable to our example and implies polynomial decay of $\mathcal{L}_{f}$ for these parameters of $\gamma$.

\subsection{The case $3/2 < \gamma \leq 2$} \label{notB}

We now consider the case of  $3/2<\gamma\leq 2$ which is related to the second case of Theorem \ref{theo:equilibrium} and Theorem \ref{theo:invariant-function-unbalanced}.
Namely, as the coefficients $h_n$ defined in
\eqref{eq:eigenfunction} satisfy $|\log h_n|\sim n^{1-\gamma}$, it follows that $\sum_{m > n} |\log h_m|^2 \sim n^{2-2\gamma}$.
Hence, $\sum_n \sum_{m > n} |\log h_m|^2 $ converges
iff $2\gamma-2>1$ which is equivalent to $\gamma>3/2$.
Therefore, if $\gamma>3/2$, the function
\begin{align}\label{L1-eigenfunction}
h_\rho(x)=
\exp\left(\sum_{i=1}^{\infty}\alpha_i x_i\right), \quad 
\end{align}
is $\rho$-almost surely well defined, where
$\rho=\otimes_{i=1}^\infty \rho_0$
is the Bernoulli product measure with parameter $1/2$ on $X= \{-1,1\}^\mathbb{N}$. With respect to $\mu$, it follows from Theorem \ref{theo:invariant-function-unbalanced} that 
\begin{align}\label{L1-eigenfunction-2}
h_\mu(x)= \exp\left(\sum_{i=1}^{\infty} \alpha_i x_i + \log \frac{\cosh(\sum_{j=1}^{i} j^{-\gamma})}{\cosh(\zeta({\gamma}))}  \right) 
\end{align}
is $\mu$-almost surely well defined. Furthermore, both functions satisfy 
 the functional equation $\mathcal{L}_{f}(h) = \lambda h$, for $\lambda =2\cosh(\zeta(\gamma))$.

\begin{theorem}\label{theorem:oscilation} Let $1 < \gamma \leq 2$ and $\mu$ as in \eqref{eq:conformal-measure-example} and $\tilde{\mu}$ as in \eqref{eq:equilibrium-example}.
	\begin{enumerate}
		\item  If $3/2 < \gamma \leq 2$, then $h_\rho(x) = \infty$ for $\mu$-a.e.  $x \in X$, and $h_\mu(x) = 0$ for $\rho$-a.e. $x \in X$. 
		\item If $\gamma>3/2$, then $\mu$ and $\tilde\mu$ are absolutely continuous, and $d\tilde \mu = h_\mu d\mu$.
		\item If $1< \gamma \leq 3/2$, then  $\mu$, $\tilde{\mu}$ and $\rho$ are pairwise singular.
	   \item  If $ 3/2 < \gamma \leq 2$, then, for any open set $A$, we have
	   \begin{align*}
	    \hbox{ess-inf}_\rho \{h_\rho(x) : x \in A\}  = \hbox{ess-inf}_\mu \{h_\mu(x) : x \in A\}   = 0, \\ 
	    \hbox{ess-sup}_\rho \{h_\rho(x) : x \in A\}  = \hbox{ess-sup}_\mu \{h_\mu(x) : x \in A\}   = \infty. 
	   \end{align*}
		In particular, neither $h_\rho$ nor $h_\mu$ can be extended to a (locally) continuous function.
	\end{enumerate}
\end{theorem}

\begin{proof} The first assertion is an application of Kolmogorov's three series theorem as in \cite[p. 87]{Luk75}. 
By a direct calculation,
\begin{align*}
E_{\mu_i} (\log h_\rho^{(i)}) &=  \int \alpha_i x \,d\mu_{i}(x)
 =  \alpha_i\frac{\exp({ \sum_{j=1}^{i} j^{-\gamma})})- \exp(-{ \sum_{j=1}^{i} j^{-\gamma})})}{2  \cosh(\sum_{j=1}^{i} j^{-\gamma})} \\
& = \alpha_i \tanh(\textstyle \sum_{j=1}^{i} j^{-\gamma}) \sim  \frac{ \tanh(\zeta(\gamma))}{(\gamma-1)} i^{1-\gamma}  \\
\mbox{Var}_{\mu_i} (\log h_\rho^{(i)}) &= \int (\alpha_i x)^2\, d\mu_i({x}) -  (\alpha_i \tanh(\textstyle \sum_{j=1}^{i} j^{-\gamma}))^2
= \alpha_i^2(1-  \tanh^2(\textstyle \sum_{j=1}^{i} j^{-\gamma}))\\
& = \frac{\alpha_i^2}{\cosh^2(\sum_{j=1}^{i} j^{-\gamma})} \sim  \frac{i^{2-2\gamma}}{(\gamma-1)^2\cosh^2(\zeta(\gamma))} 
\end{align*}
For $3/2 <  \gamma \leq 2$, it follows that $\sum_i E_{\mu_i} (\log h_\rho^{(i)}) = \infty$ and $\sum_i \mbox{Var}_{\mu_i} (\log h_\rho^{(i)}) < \infty$. This then implies that $\sum_{i=1}^\infty (\log h_\rho^{(i)} -  E_{\mu_i}(\log h_\rho^{(i)}))$ converges $\mu$-a.s. (\cite[p. 87]{Luk75}). Hence, $h_\rho=\infty$ $\mu$-a.s. 
In order to prove the  statement for  $h_\mu$ with respect to $\rho$, we apply the same arguments. Namely, the  assertion follows from   
\begin{align*}
 E_{\rho} \left(\log h_\rho^{(i)} + \log \frac{\cosh(\sum_{j=1}^{i} j^{-\gamma})}{\cosh(\zeta({\gamma}))}  \right) =  \log \frac{\cosh(\sum_{j=1}^{i} j^{-\gamma})}{\cosh(\zeta({\gamma}))} 
 \sim -\frac{\tanh{(\zeta({\gamma}))}  i^{1 -\gamma}}{\gamma -1} 
 \end{align*}
 and $\mbox{Var}_{\rho}(\log h_\mu^{(i)}) = \mbox{Var}_{\rho}(\log h_\rho^{(i)}) = \alpha_i^2 $.

The second and the third are applications of Theorem \ref{theo:invariant-function-unbalanced} and the three series theorem as in \cite[p. 88]{Luk75}. Namely, we have that
\begin{align*}
\mbox{Var}_{\mu_i}(\log h_\mu^{(i)}) =\mbox{Var}_{\mu_i} (\log h_\rho^{(i)})  \sim  \frac{i^{2-2\gamma}}{(\gamma-1)^2\cosh^2(\zeta(\gamma))}.\end{align*} 
Hence, if $\gamma\leq 3/2$, then $\log h_\mu$ does not exist in $(-\infty , \infty)$. However, $h_\mu$ exists by  Theorem \ref{theo:invariant-function-unbalanced} also in this case, but might be equal to $0$. Hence, $h_\mu=0$ and $\mu$ and $\tilde{\mu}$ are pairwise singular. Assertion (iii) then follows from the obvious fact that 
$\rho$ is singular with respect to both $\mu$ and $\tilde{\mu}$.
Furthermore, part (ii) is a consequence of Theorem \ref{theo:invariant-function-unbalanced} as $h_\mu>0$ for $\gamma>3/2$.

It remains to show the last part. We begin with the proof for $h_\rho$.
	As $A$ is open, there exist $m\in \mathbb N$ and $a_1, \ldots, a_m\in \{-1,1\}$ 
	such that $[a_1, \ldots, a_m] \subset A$. 
	In order to show that $\hbox{ess-sup}_\rho h_\rho(x) = \infty$, 
	it remains to show that,
	for all $M>0$,
	\[
	\rho\left(\left\{ x\in [a_1, \ldots, a_m] :
	\textstyle \sum_{i=1}^{\infty} \alpha_i x_i >	M \right\}
	\right) > 0.
	\]
	In order to do so, note that $\gamma \leq 2$ implies  that $\sum_{i=1}^\infty \alpha_i = \infty$. Hence, for each $M > 0$, there exists $n >m$ such that $-\alpha_{1}-\ldots-\alpha_m+\alpha_{m+1} + \ldots + \alpha_{n} > M$. For
	$\mathfrak{C} := \{x \in [a_1,\ldots,a_m]:x_{m+1} = \ldots =x_n=1\}$, we have
	\[
	\rho\left(
	\mathfrak{C}
	\cap
	\left\{
	\textstyle
	\sum_{i=n+1}^{\infty}
	\alpha_i x_i
	\geqslant 0
	\right\}
	\right)
	\leq
	\rho\left(
	\textstyle 	
	\sum_{i=1}^{\infty}
	\alpha_i x_i
	>
	M
	\right).
	\]
	Observe that the events $\mathfrak{C}$ and $
	\{
	\sum_{i=n+1}^{\infty}
	\alpha_i x_i
	\geqslant 0
	\}
	$
	are independent, that $\rho(\mathfrak{C}) = 2^{-n}$ and that, by symmetry,
	$\rho( \{\textstyle
	\sum_{i=n+1}^{\infty}
	\alpha_i x_i
	\geqslant 0 \})\geq 1/2$. Hence,
	\[ \rho\left(
	\mathfrak{C}
	\cap
	\left\{
	\textstyle
	\sum_{i=n+1}^{\infty}
	\alpha_i x_i
	\geqslant 0
	\right\}
	\right)
	=
	\rho(\mathfrak{C})
	\cdot
	\rho(
	\{\textstyle
	\sum_{i=n+1}^{\infty}
	\alpha_i x_i
	\geqslant 0 \}) \geq 2^{1-n} >0.
	\]
	Hence, $\hbox{ess-sup}_\rho \{h_\rho(x): {x \in A}\} \geq e^M$. The proof of $\hbox{ess-inf}_\rho \{h_\rho(x): {x \in A}\} =0$ follows by substituting $\mathfrak{C}$ with $\{x \in [a_1,\ldots,a_m]: x_{m+1} = \ldots =x_n=-1\}$, where $n$ is chosen such that $\alpha_{1} + \ldots + \alpha_{m} -  (\alpha_{m+1} +  \cdots + \alpha_{n}) < -M$.
	
	In order to prove the local unboundedness of $h_\mu$, we make use of (ii). Namely, in order to obtain that $\hbox{ess-sup}_\mu h_\mu(x) = \infty$, it suffices to show that, for  $[a_1, \ldots, a_m] \subset A$ and $w_n := (a_1, \ldots, a_m,1,1, \ldots,1) \in \mathcal{A}^{m+n}$, we have
	\[\lim_{n \to \infty} \frac{\int_{[w_n]} h_\mu d\mu}{\mu ([w_n])} =  \lim_{n \to \infty} \frac{\tilde \mu ([w_n])}{\mu ([w_n])} = \lim_{n \to \infty} \frac{\tilde \mu ([a_1 \cdots a_m])}{\mu ([a_1 \cdots a_m])} \frac{\tilde\mu([1,\ldots,1])}{\mu([1,\ldots,1])} = \infty, \]
where $(1,\ldots,1)$ stands for the word of length $n$ with all entries equal to one. In order to verify this condition, note that $ \log({\cosh(\sum_{l=1}^j l^{-\gamma})}/{\cosh(\zeta(\gamma))}) \sim -\tanh(\zeta(\gamma)) \sum_{l=j+1}^\infty l^{-\gamma} $. Hence,
\begin{align*}
\sum_{j=m+1}^{m+n} \log  \frac{\tilde\mu_j(1)}{\mu_j(1)} 
& =
\sum_{j=m+1}^{m+n} \log \frac{\exp \sum_{l=1}^\infty l^{-\gamma} }{2\cosh(\zeta(\gamma))}  
- \log\frac{\exp \sum_{l=1}^j l^{-\gamma}}{2\cosh(\sum_{l=1}^j l^{-\gamma})} \\
&=   \sum_{j=m+1}^{m+n} \left( \sum_{l=j+1}^\infty l^{-\gamma} + \log \frac{\cosh(\sum_{l=1}^j l^{-\gamma})}{\cosh(\zeta(\gamma))}\right)\\
& \asymp  \sum_{j=m+1}^{m+n}  (1- \tanh(\zeta(\gamma))) \sum_{l=j+1}^\infty l^{-\gamma} \xrightarrow{n\to \infty} \infty.
\end{align*}	
Hence, $\hbox{ess-sup}_\mu h_\mu(x) = \infty$. The proof of $\hbox{ess-inf}_\mu h_\mu(x) =0$ is the same.  
\end{proof}

We turn our attention to $h$ as an element of $L^1(X,\rho)$. The following results are Theorems \ref{theo:Lp-existence} and \ref{theo:uniqueness-of-h} adapted to our example.

\begin{theorem} \label{theo:example} For $\gamma > 3/2$ and $h$ as in \eqref{L1-eigenfunction}, the following holds.
	\begin{enumerate}
		\item The Ruelle operator $\mathcal{L}_{f}$ defines canonically a
		bounded linear operator on
		$L^{p}(X,\rho)$ for all $1\leqslant p\le \infty$.
		\item The function $h$ defined above belongs to $L^p(X,\rho)$ for every $1\le p<\infty$ and is the unique
		eigenfunction of $\mathcal L_f:L^p(X,\rho)\to L^p(X,\rho)$  with eigenvalue  $\lambda=2\cosh(\zeta(\gamma))$
		if and only if  the operator
		\begin{eqnarray*}
			P\phi(x_1,x_2,...) =
			\frac 1{\lambda}\left[ \phi(1,x_1,x_2,...)\mbox{\rm exp}(\zeta(\gamma))  +\phi(-1,x_1,x_2,...)
			\mbox{\rm exp}(-\zeta(\gamma))\right]
		\end{eqnarray*}
		acting on $L^1(X,\rho)$ is ergodic.
	\end{enumerate}
\end{theorem}

\affiliationone{
   L. Cioletti\\
   Departamento de Matem\'atica, UnB
   \\
   Brazil
   \email{cioletti@mat.unb.br}
   }
\affiliationtwo{
   M. Denker \\
   Mathematics Department, Pennsylvania\\ State University, U.S.A.
   \email{denker@math.psu.edu}
   }
\affiliationthree{
   A. O. Lopes \\
   Departamento de Matem\'atica Pura 
   \\ e Aplicada,  UFRGS,  Brazil.
   \email{arturoscar.lopes@gmail.com}
   }
\affiliationfour{
   M. Stadlbauer\\
   Departamento de Matem\'atica, UFRJ\\
   Brazil
   \email{manuel.stadlbauer@gmail.com}}
   
  \vfill 

\begin{thebibliography}{99}
	
	\bibitem{BCL+11} {\bibname A. T. Baraviera, L. Cioletti, A. O. Lopes, J. Mohr \and R. R. Souza}.
	On the general one-dimensional XY model: 
	positive and zero temperature, selection and non-selection.
	{\em Rev. Math. Phys.}, 23(10): 1063--1113, 2011.
	
	\bibitem{Bogachev:2007} {\bibname V. Bogachev}.
	{\em Measure theory. {V}ol. {I}}. Springer-Verlag, Berlin, 2007
	
	\bibitem{Bousch2001}	
	{\bibname T. Bousch}.
	La condition de {W}alters.
	{\em Ann. Sci. \'Ecole Norm. Sup. (4)}.
	34(2): 287--311, 2001.
	
	
	
	\bibitem{Bowen1974} {\bibname R. Bowen}.
	Some systems with unique equilibrium states.
	{\em Math. Systems Theory}, 8: 193--202, 1974.
	
	\bibitem{BFG99} {\bibname X. Bressaud, R. Fern\'andez  \and A. Galves}.
	Decay of correlations for non-H\"olderian dynamics. A coupling approach.
	{\em Electron. J. Probab.}, 4(3): 19 pp. (electronic), 1999.
	
	\bibitem{CL14} {\bibname L. Cioletti \and A. O. Lopes.}
	Interactions, specifications, probabilities and the 
	Ruelle operator in the one-dimensional lattice.
	{\em arXiv:1404.3232}, 2014.
	
	\bibitem{Climenhaga2013} {\bibname V. Climenhaga  \and D.  J. Thompson.}
	Equilibrium states beyond specification and the Bowen property.
	{\em J. Lond. Math. Soc.} (2) 87(2): 401--427, 2013.
	
	\bibitem{DU91} {\bibname M. Denker \and M. Urba\'nski}.
	On the existence of conformal measures.
	{\em Trans. Amer. Math. Soc.}, 328(2): 563--587, 1991.
	
	\bibitem{Iommi2013} {\bibname G. Iommi \and  M. Todd.}
	Transience in dynamical systems.
	{\em Ergodic Theory and Dynamical Systems}  33(5): 1450--1476, 2013.
	
	\bibitem{Lin1971} {\bibname M. Lin.}
	Mixing for Markov operators.
	{\em Z. Wahrsch. u. v. Geb.} 19(3): 231--242, 1971.
	
	\bibitem{Luk75} {\bibname E. Lukacs.}
	{\em Stochastic convergence}. Academic Press [Harcourt Brace Jovanovich, Publishers], 
	New York-London,
	second edition, 1975. Probability and Mathematical Statistics, Vol. 30.
	
	\bibitem{Markley1982} {\bibname N. G. Markley \and M. E. Paul.}
	Equilibrium states of grid functions.
	{\em Trans. Amer. Math. Soc.} 274: 169--191, 1982.
	
	
	\bibitem{Ruelle1967}
	{\bibname D. Ruelle.}
	A variational formulation of equilibrium statistical mechanics and the Gibbs phase rule.
	{\em Comm. Math. Phys.} 5: 324--329, 1967.
	
	
	\bibitem{Sarig2001} {\bibname O. Sarig}.
	Thermodynamic formalism for null recurrent potentials.
	{\em Israel J. Math.} 12:  285--311, 2001.
	
	\bibitem{SU07} {\bibname B. O. Stratmann \and M. Urba\'nski.}
	Pseudo-Markov systems and infinitely generated Schottky groups.
	{\em Amer. J. Math.}, 129(4): 1019--1062, 2007.
	
	\bibitem{Wal1975} {\bibname P. Walters.}
	A variational principle for the pressure of continuous transformations.
	{\em Amer. J. Math.} 97(4): 937--971, 1975.
	
	\bibitem{Wal75a} {\bibname P. Walters.}
	Ruelle's operator theorem and $g$-measures.
	{\em Trans. Amer. Math. Soc.} 214: 	375--387, 1975.
	
	\bibitem{Wal78} {\bibname P. Walters.}
	Invariant measures and equilibrium states for some mappings
	which expand distances, 
	{\em Trans. Amer. Math. Soc.}
	236: 121--153, 1978.
	
	
	\bibitem{Wal01} {\bibname P. Walters.}
	Convergence of the Ruelle operator for a function satisfying Bowen’s condition.
	{\em Trans. Amer. Math. Soc.}, 353(1): 327--347 (electronic), 2001.
	
	\bibitem{Wal05} {\bibname P. Walters.}
	Regularity conditions and Bernoulli properties of equilibrium states and $g$-measures.
	{\em J. London Math. Soc.} (2), 71(2): 379--396, 2005.
	
	\bibitem{Wal07} {\bibname P. Walters.}
	A natural space of functions for the Ruelle operator theorem.
	{\em Ergodic Theory Dynamical Systems}, 27(4): 1323--1348, 2007.
	
	\bibitem{Yur98} {\bibname M.  Yuri.}
	Zeta functions for certain non-hyperbolic systems and topological Markov approximations.
	{\em Ergodic Theory Dynamical. Systems}, 18(6): 1589--1612, 1998.
\end{thebibliography}
\end{document}